\documentclass{article}
\usepackage[letterpaper, margin=1in]{geometry}
\usepackage{amsmath}
\usepackage{faktor}
\usepackage{amssymb}
\usepackage{listings}
\usepackage{courier}
\usepackage{mdframed}
\usepackage{fancyhdr}
\usepackage{enumitem}
\usepackage{graphicx}
\graphicspath{{images/}}
\usepackage[utf8]{inputenc}
\usepackage[english]{babel}
\usepackage{amsthm}
\usepackage[T1]{fontenc}

\theoremstyle{plain}
\newtheorem{thm}{Theorem}[section]
\newtheorem{lem}[thm]{Lemma}
\newtheorem{prop}[thm]{Proposition}
\newtheorem{cor}[thm]{Corollary}

\theoremstyle{defn}
\newtheorem{defn}[thm]{Definition}

\theoremstyle{remark}

\usepackage{mathabx}
\usepackage{color}

\usepackage{color,soul}
\usepackage{enumitem}
 
\begin{document}

\begin{titlepage}
    \begin{center}
        \vspace*{1cm}
        \Huge
        
        \textbf{A Report on Hausdorff Compactifications of $\mathbb{R}$}
        \vspace{3.0cm}
        \Large
        
        \textbf{Arnold Tan Junhan} 
        \vspace{2.0cm}
        \Large
        
        \textbf{Michaelmas 2018 Mini Projects: Analytic Topology}
        \vfill
        \vspace{0.8cm}
        \Large
        University of Oxford   
    \end{center}
\end{titlepage}
\tableofcontents
%we should have consistent capitalisation of section names
\section{Abstract}

The goal of this report is to investigate the variety of Hausdorff compactifications of $\mathbb{R}$. The Alexandroff one-point compactification, the two-point compactification $[-\infty,\infty]$, and the Stone-$\check{\text{C}}$ech compactification are all clearly different. The ultimate aim is to show that there are in fact uncountably many. An intermediate aim is to exhibit one compactification of $\mathbb{R}$ different from all the compactifications already mentioned.
\\

\noindent We will often just write $\delta X$ to refer to a compactification $\langle l, \delta X \rangle$ of a space $X$. We will compare two $T_2$ compactifications of a space $X$ by writing $\langle l_1, \delta_1 X \rangle \leq \langle l_2, \delta_2 X \rangle$ to mean that there is a continuous function $L: \delta_2 X \rightarrow \delta_1 X$ such that $L \circ l_2 = l_1$. (Such a function will automatically be onto.) It is not hard to see that if $\delta_1 X \leq \delta_2 X$ and $\delta_2 X \leq \delta_1 X$ then $\delta_1 X$ and $\delta_2 X$ are homeomorphic as topological spaces.
\\

\noindent Let us declare two compactifications $\langle l_1, \delta_1 X \rangle$ and $\langle l_2, \delta_2 X \rangle$ to be \textit{equivalent} if $\delta_1 X \leq \delta_2 X$ and $\delta_2 X \leq \delta_1 X$. Then $\leq$ gives us a partial ordering on the set of equivalence classes of compactifications. This will be useful for us towards the end of the report, where we shall apply Zorn's Lemma to this poset of equivalence classes.
\\

\noindent For that purpose, let us also recall here that an element $p \in P$ of a poset $(P, \leq)$ is \textit{maximal} if whenever we have $q \in P$ with $p \leq q$, then $p=q$. (When the equivalence class of a compactification is maximal -- with respect to $\leq$, among all compactifications with some given property -- we will simply say the compactification is maximal.) On the other hand $p \in P$ is a \textit{greatest} element if $q \leq p$ for all $q \in P$. Writing $p < q$ to mean $p \leq q$, $p \neq q$ (and writing $p \nless q$ otherwise), we see that $p$ is maximal iff $p \nless q$ for all $q \in P$. A greatest element in a poset is unique and certainly maximal, however we may have several different maximal elements. A \textit{chain}, or \textit{linearly ordered set}, is a poset $(P, \leq)$ in which we have comparability of elements: for all $p,q \in P$, either $p \leq q$ or $q \leq p$. In a chain, the notions of maximal and greatest element do coincide.
\\

\section{Compactifications via their characterising properties}
The reader is surely familiar with the idea that the essence of the \textit{Stone-$\check{\text{C}}$ech compactification} $ \langle h, \beta X \rangle$ can be captured via a certain characterizing property. We run through the steps of showing this, and then, borrowing some of these ideas, we will exhibit a compactification of $\mathbb{R}$ that turns out to be different from $ \langle h, \beta X \rangle$.

\begin{defn}
Let $X$ be a (nonempty) Tychonoff space.
\\ Let $\{f_\lambda: \lambda \in \Lambda \}$ be a list of all bounded continuous functions from $X$ to $\mathbb{R}$.
\\ For each $\lambda$, let $I_\lambda$ be the smallest closed interval such that $ran(f_\lambda) \subseteq I\lambda$. That is, let $I_\lambda=[\text{inf ran}(f_\lambda),\text{sup ran}(f_\lambda)]$.
\\ Let $Y = \prod_{\lambda \in \Lambda} I_\lambda$ be the Tychonoff product of the $I_\lambda$.
\\ Define $h: X \rightarrow Y$ such that for each $\lambda \in \Lambda$, $h(x)(\lambda) = f_\lambda (x)$. Let $\beta X = \text{cl}^Y(h(X))$.
\\ Define the \textbf{Stone-$\check{\text{C}}$ech compactification} of $X$ to be $ \langle h, \beta X \rangle$.
\end{defn}

\noindent Let us briefly check that this is indeed a $T_2$ compactification of $X$:
\begin{itemize}
    \item $Y = \prod_{\lambda \in \Lambda} I_\lambda$, which is compact (by Tychonoff's Theorem) and $T_2$, since this is true for each of the $I_\lambda$. Therefore, since $\beta X$ is a subspace of $Y$, it is $T_2$; since it is closed in $Y$, it is compact.
    \item $h$ is injective. Suppose we have distinct points $x, y \in X$. Since $X$ is Tychonoff (and hence $\{y\}$ is closed), there is a continuous function $f: X \rightarrow [0,1]$ such that $f(x)=0, f(\{y\})=\{1\}$. $f$ is bounded, so there is $\lambda \in \Lambda$ with $f=f_\lambda$. Then, $h(x)(\lambda)=f_\lambda (x) = 0 \neq 1 = f_\lambda (y) = h(y)(\lambda) $, so $h(x) \neq h(y)$;
    \item $h$ is continuous. A subbasic open set in $Y$ has the form $U_\lambda \times \prod _{\mu \neq \lambda} I_\mu$, where $U_\lambda$ is open in $I_\lambda$. Set $U = (U_\lambda \times \prod _{\mu \neq \lambda} I_\mu )\cap \beta X$, then
$$h ^{-1} (U) = \{x \in X: h (x) \in U\}= \{x \in X: h (x)(\lambda) \in U_\lambda \}= \{x \in X: f_\lambda (x) \in U_\lambda \} =  f_\lambda ^{-1} (U_\lambda),$$
and this is open, since $f_\lambda$ is continuous;
    \item $h^{-1}$ is continuous. It is enough to see that whenever $x \in U$, where $U$ is open in $h(X)$, there is an open $V \ni h(x)$ in $h(X)$ such that $h^{-1} (V) \subseteq U$. Well, since $x \notin X \backslash U$, which is closed, and $X$ is Tychonoff, there is some continuous function $f : X \rightarrow [0,1]$ such that $f(x) = 0$ and $f(X \backslash U) = \{1\}$. $f$ is a bounded continuous function from $X$ to $\mathbb{R}$, so there is $\lambda \in \Lambda$ with $f = f_\lambda$. Hence $f_\lambda (x) = 0$ and $f_\lambda (X \backslash U) = \{1\}$.
    \\ Note that $V_\lambda := [0,1)$ is open in $[0,1] = I_\lambda$, so $V := (V_\lambda \times \prod_{\mu \neq \lambda} I_\mu) \cap h(X)$ is open in $h(X)$. We have
    $$h^{-1} (V) = \{x \in X: h(x) \in V\} = \{ x \in X: h(x) (\lambda) \in V_\lambda \} = \{x \in X: f_\lambda (x) \in [0,1)\} = f_\lambda ^{-1} [0,1),$$
    but $x \in f_\lambda ^{-1} [0,1) \subseteq U$, so $x \in h^{-1} (V) \subseteq U$;
    \item $\text{cl}^{\beta X} (h(X)) = \beta X$ holds. $\beta X$ is the smallest closed set in $Y$ containing $h(X)$, so it is the smallest closed set in $\beta X$ containing $h(X)$, because $\beta X$ is closed in $Y$ by construction.
\end{itemize}

\begin{lem} \label{sc1}
Let $X$ be a Tychonoff space, and $I$ be a closed bounded interval in $\mathbb{R}$. Let $f: X \rightarrow I$ be continuous.
Then there exists a continuous function $\beta f : \beta X \rightarrow I$ such that $\beta f \circ h = f$.
\end{lem}

\begin{proof}
$f$ is bounded and continuous, so there is some $\lambda \in \Lambda$ such that $f=f_\lambda$.
\\ Define $\beta f : \beta X \rightarrow I, y \mapsto y(\lambda)$. This is a projection, so it is continuous.
\\ Furthermore, for all $x \in X$, we have $\beta f \circ h (x) = h(x) (\lambda)= f_\lambda (x)= f(x)$.
\end{proof}

\begin{lem} \label{sc2}
Let $X$ be a Tychonoff space, and $Z = \prod_{\mu \in M} I_\mu$ be a product of closed bounded intervals in $\mathbb{R}$. Let $f: X \rightarrow Z$ be continuous.
Then there exists a continuous function $\beta f : \beta X \rightarrow Z$ such that $\beta f \circ h = f$.
\end{lem}

\begin{proof}
Define $f^\mu : X \rightarrow I_\mu, x \mapsto f(x)(\mu)$. $f^\mu = \pi_\mu \circ f$, so it is continuous.
Apply Lemma \ref{sc1} to see that there exists a continuous function $\beta f^\mu : \beta X \rightarrow I_\mu$ such that $\beta f^\mu \circ h = f^\mu$.
\\ Now define $\beta f : \beta X \rightarrow Z$ such that for all $\mu$, $\beta f (x) (\mu) = \beta f^\mu (x)$. That is, $\pi_\mu \circ \beta f = \beta f^\mu$.
\\ It remains to see that $\beta f$ is continuous.
\\ A subbasic open set in $Z$ has the form $U=U_\mu \times \prod _{\nu \neq \mu} I_\nu$, where $U_\mu$ is open in $I_\mu$. We have
$$\beta f ^{-1} (U) = \{x \in \beta X: \beta f (x) \in U\}= \{x \in \beta X: \beta f (x)(\mu) \in U_\mu \}= \{x \in \beta X: \beta f^\mu (x) \in U_\mu \} = (\beta f^\mu) ^{-1} (U_\mu).$$
This is open, since $\beta f^\mu$ is continuous.
\end{proof}

\begin{lem} \label{embed}
Any Tychonoff space $X$ can be embedded in a product of closed bounded intervals.
\end{lem}

\begin{proof}
$\beta X$ is a subset of such a product!
\end{proof}

\begin{thm}[The Stone-$\check{\text{C}}$ech Property]
Let $X$ be a Tychonoff space.
\\ Say a compactification $(k, \gamma X)$ of $X$ has the \textbf{Stone-$\check{\text{C}}$ech property} if whenever $K$ is a compact $T_2$ space and $f:X \rightarrow K$ is continuous, there exists a continuous map $\gamma f : \gamma X \rightarrow K$ such that $\gamma f \circ k = f$.
\\ ($\gamma f$ will automatically be unique, since it is already determined on the dense set $k(X) \subseteq \gamma X$.)
\\ Then $(h, \beta X)$ has the Stone-$\check{\text{C}}$ech property.
\end{thm}

\begin{proof}
Since $K$ is compact $T_2$, it is Tychonoff. By Lemma \ref{embed}, without loss of generality there is a product $Z = \prod_{\mu \in M} I_\mu$ of closed bounded intervals such that $K \subseteq Z$.
Viewing $f$ as a continuous function $X \rightarrow Z$, Lemma \ref{sc2} gives us a continuous function $\beta f : \beta X \rightarrow Z$ such that $\beta f \circ h = f$. It only remains to see that the image of $\beta f$ lies in $K$.
\\ K is compact in the Hausdorff space $Z$, hence $K$ is closed in $Z$, so $(\beta f)^{-1} (K)$ is closed in $\beta X$. Also, $f(X) \subseteq K$ implies $h(X) \subseteq (\beta f)^{-1} (K)$. Since $h(X)$ is dense in $\beta X$, we must have $(\beta f)^{-1} (K) = \beta X$.
\end{proof}

\begin{thm} \label{largest}
If $(k, \gamma X)$ is a Hausdorff compactification of $X$ that has the Stone-$\check{\text{C}}$ech property, then $ \langle k', \gamma' X \rangle \leq \langle k, \gamma X \rangle$ for any other compactification $ \langle k', \gamma' X \rangle$.
\end{thm}

\begin{proof}
Take $K= \gamma' X$ and $f = k'$ in the definition of $(k, \gamma X)$ having the Stone-$\check{\text{C}}$ech property, to see that there exists a continuous map $\gamma h : \gamma X \rightarrow \gamma' X$ such that $\gamma k' \circ k = k'$. This is precisely the statement that $ \langle k', \gamma' X \rangle \leq \langle k, \gamma X \rangle$.
\end{proof}

\noindent Since the Stone-$\check{\text{C}}$ech compactification has the Stone-$\check{\text{C}}$ech property, we deduce:

\begin{cor}
$\langle h, \beta X \rangle$ is the largest compactification of $X$.
\end{cor}

\noindent On the other hand, we could set $\langle k', \gamma' X \rangle$ in Theorem \ref{largest} to be the Stone-$\check{\text{C}}$ech compactification, to get another corollary:

\begin{cor} \label{sc3}
If $\langle k, \gamma X \rangle$ is a Hausdorff compactification of $X$ that has the Stone-$\check{\text{C}}$ech property, then $ \langle h, \beta X \rangle \leq \langle k, \gamma X \rangle$.
\end{cor}

\noindent This says that the Stone-$\check{\text{C}}$ech compactification is the smallest one having the Stone-$\check{\text{C}}$ech extension property. Suppose now that we consider the problem of extending a given family of bounded continuous functions on $X$, rather than all bounded continuous functions.
\\

\noindent For example, suppose we are asked to construct a Hausdorff compactification $\langle k, \gamma \mathbb{R} \rangle $ of $\mathbb{R}$ that has the following property: whenever $f: \mathbb{R} \rightarrow \mathbb{R}$ is of the form $f(x) = \cos(nx)$ for some $n \in \mathbb{Z}$, there exists a continuous function $\gamma f : \gamma  \mathbb{R} \rightarrow  \mathbb{R}$ such that $\gamma f \circ k = f$.
\\

\noindent We give a construction of such a compactification $\langle k, \gamma \mathbb{R} \rangle $, by altering that of $\langle h, \beta X \rangle$. We note beforehand that $\cos(nx)=\cos(-nx)$ for each $n \in \mathbb{Z}$, and the constant function $\cos(0)=1$ extends trivially to any compactification, so we need only consider $n \geq 1$.

\begin{prop}
Consider a set $\{f_n: n \in \mathbb{N} \}$ of functions from $\mathbb{R}$ to $[-1,1]$, where $$f_0 (x)= \tanh(x), \ \ \ \ f_n (x) = \cos(x) \ \forall n \geq 1.$$
\\ Let $Y = \prod_{n \in \mathbb{N}} [-1,1]$.
\\ Define $k: \mathbb{R} \rightarrow Y$ such that for each $n \in \mathbb{N}$, $k(x)(n) = f_n (x)$. Let $\gamma \mathbb{R} = \text{cl}^Y(k(\mathbb{R}))$.
\\ Then $ \langle k, \gamma \mathbb{R} \rangle$ is a compactification of $\mathbb{R}$.
\end{prop}

\begin{proof}
We check this is a compactification.
\begin{itemize}
    \item $Y = \prod_{n \in \mathbb{N}} [-1,1]$ is compact (by Tychonoff's Theorem) and $T_2$, since this is true for $[-1,1]$. Therefore, since $\gamma X$ is a subspace of $Y$, it is $T_2$; since it is closed in $Y$, it is compact.
    
    \item $k$ is injective. Suppose we have distinct points $x, y \in \mathbb{R}$. $f_0(x)=\tanh(x)$ is strictly monotone, hence injective. Therefore, $k(x)(n)=f_0 (x) \neq f_0 (y) = k(y)(n) $, so $k(x) \neq k(y)$;
    
    \item $k$ is continuous. A subbasic open set in $Y$ has the form $U_n \times \prod _{m \neq n} [-1,1]$, where $U_n$ is open in $[-1,1]$. Set $U = (U_n \times \prod _{m \neq n} I_m )\cap \gamma \mathbb{R}$, then
$$k ^{-1} (U) = \{x \in \mathbb{R}: k (x) \in U\}= \{x \in \mathbb{R}: k (x)(n) \in U_n \}= \{x \in \mathbb{R}: f_n (x) \in U_n \} =  f_n ^{-1} (U_n),$$
and this is open, since $f_n$ is continuous;

    \item $k^{-1}$ is continuous. It is enough to see that whenever $x \in U$, where $U$ is open in $k(\mathbb{R})$, there is an open $V \ni k(x)$ in $k(\mathbb{R})$ such that $k^{-1} (V) \subseteq U$. Well, since $x \notin \mathbb{R} \backslash U$, which is closed, $f_0(x)=\tanh(x)$ is such that $f_0 (x) \notin \text{cl}^{[-1,1]}(f_0 (\mathbb{R} \backslash U))$. %(If there was a sequence of points in $f_0 (\mathbb{R} \backslash U)$ converging to $f_0 (x)$ in $[-1,1]$, then by injectivity of $f_0$ we would have a sequence in $\mathbb{R} \backslash U$ converging to $x$)
    \\ Then $f_0(x) \in V_0 := [-1,1] \backslash \text{cl}^{[-1,1]}(f_0 (\mathbb{R} \backslash U))$, which is open in $[-1,1]$.
    \\ The set $V := (V_0 \times \prod_{n \neq 0} [-1,1]) \cap k(\mathbb{R})$ is open in $k(\mathbb{R})$, and we have
    $$k^{-1} (V) = \{x \in \mathbb{R}: k(x) \in V\} = \{ x \in \mathbb{R}: k(x) (0) \in V_0 \} = \{x \in \mathbb{R}: f_0 (x) \in V_0 )\} = f_0 ^{-1} (V_0),$$
    so $x \in f_0 ^{-1} V_0 = k^{-1}(V)$ shows that $k(x) \in V$. \\ Finally, note that $k^{-1}(V) \subseteq U$, since $f_0 (\mathbb{R} \backslash U) \subseteq [-1,1] \backslash V_0$ implies $f_0^{-1}(V_0)=\mathbb{R} \backslash f_0^{-1}([-1,1] \backslash V_0) \subseteq U$;
    \item $\text{cl}^{\gamma \mathbb{R}} (k(\mathbb{R})) = \gamma \mathbb{R}$ holds.
\end{itemize}

\end{proof}

\noindent Next, let us show that each $f_n$ does extend continuously onto $ \langle k, \gamma \mathbb{R} \rangle$.

\begin{lem} \label{extendo1}
Let $f_n: \mathbb{R} \rightarrow \mathbb{R}, \ x \mapsto \cos(nx)$, where $n \geq 1 $.
Then there exists a continuous function $\gamma f_n : \gamma \mathbb{R} \rightarrow [-1,1]$ such that $\gamma f_n \circ k = f_n$.
\end{lem}

\begin{proof}
Simply define $\gamma f_n : \gamma \mathbb{R} \rightarrow [-1,1], y \mapsto y(n)$. This is a projection, so it is continuous.
\\ Furthermore, for all $x \in \mathbb{R}$, we have $\gamma f_n \circ k (x) = k(x) (n)= f_n (x)$.
\end{proof}

\noindent This already gives us the result that whenever $f: \mathbb{R} \rightarrow \mathbb{R}$ is of the form $f(x) = \cos(nx)$ for some $n \in \mathbb{Z}$, there exists a continuous function $\gamma f : \gamma  \mathbb{R} \rightarrow  \mathbb{R}$ such that $\gamma f \circ k = f$.

\noindent Next, we show that $ \langle k, \gamma \mathbb{R} \rangle$ is the smallest compactification to which $f_n$ extends continuously for each $n \geq 0$.

\begin{prop}
Suppose $\langle l, \delta \mathbb{R} \rangle $ is a Hausdorff compactification of $\mathbb{R}$ that has the following property: for each $n \geq 0 $, there exists a continuous function $\delta f_n : \delta \mathbb{R} \rightarrow [-1,1]$ such that $\delta f_n \circ l = f_n$.
Then $ \langle k, \gamma \mathbb{R} \rangle \leq \langle l, \delta \mathbb{R} \rangle$.
\end{prop}

\begin{proof}
Define $F : \delta \mathbb{R} \rightarrow \gamma \mathbb{R} $ as follows: for each $y \in \delta \mathbb{R}$ and $n \in \mathbb{N}$, let $F(y)(n)= \delta f_n (y)$. Clearly, $F \circ l = k$, since for all $x \in \mathbb{R}$,
$$F \circ l (x) (n) = F (l (x)) (n) = \delta f_n (l(x)) = f_n(x)=  k (x) (n).$$

\noindent It remains to see that $F$ is continuous.
\\ Recall that $\gamma \mathbb{R} \subseteq \prod_{n \in \mathbb{N}} [-1,1]$, and a subbasic open set in $\prod_{n \in \mathbb{N}}[-1,1]$ has the form $U_n \times \prod _{m \neq n} I_m$, where $U_n$ is open in $[-1,1]$. Let $U = (U_n \times \prod _{m \neq n} I_m) \cap \gamma \mathbb{R} $, then
$$F ^{-1} (U) = \{y \in \delta \mathbb{R}: F (y) \in U\}= \{y \in \delta \mathbb{R}: F (y) (n) \in U_n\} = \{y \in \delta \mathbb{R}: \delta f_n (y) \in U_n \} =(\delta f_n)^{-1}(U_n).$$
This is open, since $\delta f_n$ is continuous.
\end{proof}

\noindent Note that this proposition does not quite tell us that $ \langle k, \gamma \mathbb{R} \rangle$ is the smallest compactification to which $f(x)=\cos(nx)$ extends continuously for each $n \in \mathbb{Z}$, because among the $f_n$ is the function $f_0(x)=\tanh(x)$, which we added to the family in order to construct $\gamma \mathbb{R}$. We did this so that the family would \textit{separate points and closed sets}; for a more general construction see Folland (1999).
\\

\noindent Nevertheless, we shall show in the next section that this compactification is genuinely different from the ones we have seen before.

\iffalse
\begin{defn}
Let $X$ be a topological space. Let $\mathfrak{F}$ be a subset of $C(X,I)$, the set of continuous functions from $X$ to $[0,1]$. Say $\mathfrak{F}$ \textbf{separates points and closed sets} if for each closed subset $E \subseteq X$ and each $x \in X \backslash E$ there is some $f \in \mathfrak{F}$ such that $f(x) \notin \text{cl}^{[0,1]}(f(E))$. 
\end{defn}

\noindent Observe that each such family $\mathfrak{F}$ gives rise to a map $e: X \rightarrow \prod_{f \in \mathfrak{F}} I$.

\begin{prop}[Folland, 1999]
Let $X$ be Tychonoff, and $\mathfrak{F} \subseteq C(X,I)$ separate points and closed sets.
\end{prop}

\noindent Each such family $\mathfrak{F}$ gives rise to a compactification of $X$. In particular, the Stone-$\check{\text{C}}$ech compactification is the special case of when $\mathfrak{F}$ is the set of all continuous functions from $X$ to $[0,1]$, meaning it really is a $T_2$ compactification.
\fi
\section{A genuinely new compactification}

The compactification we have just constructed is genuinely different from any of the one-point, two-point, or Stone-$\check{\text{C}}$ech compactification of $\mathbb{R}$. It cannot be the one-point or two-point compactification, because the function $\mathbb{R} \rightarrow \mathbb{R}, x \mapsto \text{cos}(x)$ does not extend continuously to either of these:

\begin{prop}
Let $f: \mathbb{R} \rightarrow \mathbb{R}, x \mapsto \text{cos}(x)$. There is no continuous function extending $f$ to either the Alexandroff one-point compactification or the two point compactification of $\mathbb{R}$.
\end{prop}

\begin{proof}
Suppose for a contradiction we did have such an extension $\tilde{f}$.
\\ \noindent One way to write the one-point compactification is as $\langle i_1, S^1 \rangle$ where $i_1: \mathbb{R} \rightarrow S^1, x \mapsto (\frac{2x}{1+x^2},\frac{x^2-1}{1+x^2})$. Note that $\text{lim}_{n \rightarrow \infty} i_1(n)=(0,1)$. Hence, by the continuity of $\tilde {f}$ we would have

$$\tilde {f} (0,1) = \tilde {f} (\text{lim}_{n \rightarrow \infty} i_1(n))=  \text{lim}_{n \rightarrow \infty} (\tilde {f} \circ i_1)(n) =\text{lim}_{n \rightarrow \infty} \text{cos}(n), $$

\noindent but this does not exist in $\mathbb{R}$.
\\

\noindent Similarly, we may write the two-point compactification as $\langle i_2, [-1,1] \rangle$ where $i_2: \mathbb{R} \rightarrow [-1,1], x \mapsto \text{tanh}(n)$. Now $\text{lim}_{n \rightarrow \infty} i_2(n)=1$, so

$$\tilde {f} (1) = \tilde {f} (\text{lim}_{n \rightarrow \infty} i_2(n))=  \text{lim}_{n \rightarrow \infty} (\tilde {f} \circ i_2)(n) =\text{lim}_{n \rightarrow \infty} \text{cos}(n) $$

\noindent again contradicts that this limit does not exist.

\end{proof}

\noindent To show that $\gamma \mathbb{R}$ is not homeomorphic to $\beta \mathbb{R}$, we will show that the former is metrisable while the latter is not.
\\

\noindent The following is a standard result.

\begin{lem}
A countable product of metric spaces is metrisable.
\end{lem}

\begin{proof}
The result is easy for finite products. (Alternatively, if you like, it is deducible from the case of countably infinite products, by setting all-but-finitely-many of the factors to be singletons.)
\\

\noindent Let $\{(X_n, d_n): n \in \mathbb{N}\}$ be a countably infinite family of metric spaces.

\begin{itemize}
    \item[] \textbf{Claim}: We may assume each $d_n$ is bounded above by $1$.
    \item[] \textbf{Proof}: To prove the claim, it is enough to see that any metric $d$ on any space $X$ has an equivalent metric $d'$ defined by $d'(x,y)=\text{min}\{1,d(x,y)\}$. This is a metric on $X$:
    \begin{itemize}
        \item it is non-negative, and zero if and only if $x=y$;
        \item it is symmetric in its variables;
        \item $\text{min}\{1,d(x,z)\} \leq \text{min}\{1,d(x,y)\} +\text{min}\{1,d(y,z)\}$.
        \\ If $d(x,y), d(y,z) \leq 1$, then $$\text{min}\{1,d(x,z)\} \leq d(x,z) \leq d(x,y) + d(y,z)=\text{min}\{1,d(x,y)\} +\text{min}\{1,d(y,z)\}.$$
        Otherwise, without loss $d(x,y)> 1$, then $$\text{min}\{1,d(x,z)\} \leq 1 = \text{min}\{1,d(x,y)\} \leq \text{min}\{1,d(x,y)\} +\text{min}\{1,d(y,z)\}.$$
    \end{itemize}
    $d'$ induces the same topology as $d$ does on $X$.
    Indeed, wite $B^d_r(x)$ and $B^{d'}_r(x)$ respectively for the open balls of radius $r$ centered at $x$, with respect to $d$ and $d'$ respectively. Since $d' \leq d$, we certainly have $B^d_r(x) \subseteq B^{d'}_r(x)$ for all $r >0, x\in X$, so the topology induced by $d$ is finer than that induced by $d'$. On the other hand, for all $r >0, x\in X$, we have $B^{d'}_{r'}(x) \subseteq B^{d}_r(x)$ where $r'=\text{min}\{1,r\}$. Indeed, suppose $y \in B^{d'}_{r'}(x)$. Then $\text{min}\{1,d(x,y)\} < \text{min}\{1,r\}$, so $d(x,y) < r$.
\end{itemize}

\noindent By the claim, we may assume each $d_n$ is bounded above by $1$, so it makes sense to define, for $x,y \in \prod_{n \in \mathbb{N}}X_n$, $$d(x,y) = \Sigma_{n \in \mathbb{N}} \frac{d_n(x(n),y(n))}{2^n},$$
since this series converges to a value no greater than the convergent series $\Sigma_{n \in \mathbb{N}} \frac{1}{2^n}=2$.

\begin{itemize}
    \item[] \textbf{Claim}: $d$ defined above is a metric on the product space $\prod_{n \in \mathbb{N}}X_n$.
    \item[] \textbf{Proof}:
    \begin{itemize}
        \item it is non-negative, and zero if and only if every term in the series is zero, if and only if $x$ and $y$ agree on every component, if and only if $x=y$;
        \item it is symmetric in its variables;
        \item $\Sigma_{n \in \mathbb{N}} \frac{d_n(x(n),z(n))}{2^n} \leq \Sigma_{n \in \mathbb{N}} \frac{d_n(x(n),y(n))}{2^n} +\Sigma_{n \in \mathbb{N}} \frac{d_n(y(n),z(n))}{2^n}$ follows immediately from the triangle inequalities for the individual $d_n$.
    \end{itemize}
\end{itemize}

\noindent It remains to check that this metric induces the usual product topology on $\prod_{n \in \mathbb{N}}X_n$.
\\

\noindent Given $r>0, N \in \mathbb{N},$ and $x,y \in \prod_{n \in \mathbb{N}}X_n$, we certainly have $d_N(x(N),y(N)) < r$ whenever $d(x,y) < \frac{r}{2^N}$. Therefore, the projections $\pi_N: (\prod_{n \in \mathbb{N}}X_n,d) \rightarrow (X_N,d_N)$ are continuous with respect to these metrics. Therefore the topology $\tau_d$ induced by $d$ on the product space is finer than the Tychonoff topology $\tau$. One way to see this is via the universal property of the product: the projection maps $\pi_N: (\prod_{n \in \mathbb{N}}X_n,d) \rightarrow (X_N,d_N)$ give rise to a unique continuous map $i: (\prod_{n \in \mathbb{N}}X_n,\tau_d) \rightarrow (\prod_{n \in \mathbb{N}}X_n,\tau)$ such that $i \circ \pi_N = \pi_N$ for each $N$. Of course, setting $i$ to be the identity map satisfies this equation, and therefore we must have that the identity is continuous as a map $(\prod_{n \in \mathbb{N}}X_n,\tau_d) \rightarrow (\prod_{n \in \mathbb{N}}X_n,\tau)$. In particular, taking the preimage of each open set under the identity map, we see that $\tau \subseteq \tau_d$.
\\

\noindent On the other hand, we show that any open set $U$ in $(\prod_{n \in \mathbb{N}}X_n,\tau_d)$ is also open in the Tychonoff topology. Let $x \in U$. There is some $r >0$ with $B^d_r(x) \subseteq U$. Choose some $k$ large enough so that $\Sigma_{n=k+1}^\infty \frac{1}{2^n}=\frac{1}{2^k}<\frac{r}{2}$.
\\ For each $n \in \{0, \ldots, k\}$, define $U_n = B^{d_n}_{r / 4}(x(n))$.
Then,
$$x \in \bigcap_{n=0}^k\pi_n^{-1}(U_n) \subseteq B^d_r(x) \subseteq U.$$
Indeed, whenever $y \in \bigcap_{n=0}^k\pi_n^{-1}(U_n)$, we have $d_n(x(n),y(n)) < r / 4$ for each $n \in \{0, \ldots, k\}$, so
$$d(x,y) = \Sigma_{n=0}^k \frac{d_n(x(n),y(n))}{2^n} + \Sigma_{n=k+1}^\infty \frac{d_n(x(n),y(n))}{2^n} < \frac{r}{2} + \frac{r}{2} = r.$$

\noindent Since $\bigcap_{n=0}^k\pi_n^{-1}(U_n) \in \tau$, we have shown that $U$ is open in the Tychonoff topology, as required!
\end{proof}

\begin{cor}
$\gamma \mathbb{R}$ is metrisable.
\end{cor}

\begin{proof}
$\gamma \mathbb{R}$ can be embedded into the product $\prod_{n \in \mathbb{N}} [-1,1]$, which by the lemma above can be given a metric space structure. Identifying $\gamma \mathbb{R}$ with its image in the product space, it will inherit the subspace metric induced by the metric on the product space.
\end{proof}

\begin{lem} \label{nomax}
A non-compact Tychonoff space has no maximal metrisable $T_2$ compactifications.
\end{lem}

\begin{proof}
Suppose $X$ is a non-compact metric space, and $\langle m , \eta X \rangle$ is a metrisable $T_2$ compactification. We construct another compactification that is strictly larger. $m(X)$ is homeomorphic to $X$, hence non-compact, hence $m(X) \neq \eta X$. Pick any $x \in \eta X \backslash m(X)$. Since $\text{cl}^{\eta X}(m(X))=\eta X$, there is a sequence $(x_n)$ of distinct points in $m(X)$ converging to $x$ in the metric $d$ on $\eta X$. (The open ball $B_1(x)$ must meet $m(X)$ at some point $x_1$; the open ball $B_{\text{min}\{2^{-i},d(x,x_i)\}}(x)$ must meet $m(X)$ at some point $x_{i+1}$ for each $i \geq 1$. In this way we construct an infinite sequence of distinct points of $m(X)$ whose distance to $x$ tends to $0$.)
\\

\noindent Consider the disjoint subsets $S_0= \{x_i: i\text{ is even}\}$ and $S_1= \{x_i: i\text{ is odd}\}$ of $m(X)$. Each $S_i$ is closed in $m(X)$, since no any sequence in $S_i$ has a limit in $m(X)$. (If the limit of such a sequence existed, it would have to be $x$, but this is in $\eta X \backslash m(X)$.) Since $m(X)$ is a subset of the metric space $\eta X$, it is metrisable and hence normal, so by Urysohn's Lemma there exists a continuous function $F: m(X) \rightarrow [0,1]$ such that $F(S_0)=\{0\}, F(S_1)=\{1\}$. This function does not extend continuously to $\eta X$. For if it did, then we would have $$F(x)=F(\text{lim}_{n \rightarrow \infty}x_{2n})=\text{lim}_{n \rightarrow \infty}F(x_{2n})=0,$$ and similarly $$F(x)=F(\text{lim}_{n \rightarrow \infty}x_{2n+1})=\text{lim}_{n \rightarrow \infty}F(x_{2n+1})=1,$$ which taken together produce an obvious contradiction.
\\

\noindent Consider the function $m': X \rightarrow \eta X \times [0,1], s \mapsto (m(s),F(s))$. This is continuous since both of its components are continuous. Then $\langle m', \tilde{X} \rangle $ where $\tilde{X} = \text{cl}^{\eta X \times [0,1]} (m'(X)) $ is a $T_2$ compactification of $X$:
\begin{itemize}
    \item $\tilde{X}$ is compact and $T_2$, since it is a closed subspace of the compact $T_2$ space $\eta X \times [0,1]$;
    \item $m'$ is injective since its first component is injective;
    \item $m'$ is continuous;
    \item $m'^{-1}$ is continuous as the composition of the first projection $\pi_1: m'(X) \rightarrow \pi_1(m'(X))$ and the map $m^{-1}: m(X) \rightarrow X$;
    \item $\text{cl}^{\tilde{X}} (m'(X)) = \tilde{X}$ holds.
\end{itemize}

\noindent This compactification is larger than $\langle m , \eta X \rangle$, because there exists a continuous function $\pi_1: \tilde{X} \rightarrow \eta X$ such that $\pi_1 \circ m' = m$; this is simply the first projection $\pi_1: (z,t) \mapsto z$.
\\

\noindent On the other hand, $F$ extends continuously to $\tilde{X}$; consider $\tilde{F}: \tilde{X} \rightarrow [0,1], \ (z,t) \mapsto t$. This is just the second projection, so it is continuous, and we have $\tilde{F} \circ m'= F$. Since $F$ did not extend continuously to $\eta X$, we conclude that there is no homeomorphism from $\eta X$ to $\tilde{X}$ (or else we could compose $\tilde{F}$ with such a homeomorphism to get an extension of $F$ to $\eta X$). In particular, $\tilde{X}$ is a strictly larger compactification of $X$.

\end{proof}

\begin{cor}
For any non-compact Tychonoff space $X$, the Stone-$\check{\text{C}}$ech compactification $\langle h,\beta X \rangle$ is not metrisable.
\end{cor}

\begin{proof}
$\beta X$ is maximal among \textit{all} compactifications, hence if it were metrisable it would be maximal among all metrisable compactifications.
\end{proof}

\begin{cor}
$\beta \mathbb{R}$ is not metrisable.
\end{cor}

\begin{proof}
$\mathbb{R}$ is a non-compact metric space!
\end{proof}

\noindent Now, clearly $\beta \mathbb{R}$ was homeomorphic to $\gamma \mathbb{R}$, since one is metrisable and the other is not. We therefore obtain our desired result:
\begin{cor}
$\gamma \mathbb{R}$ is not homeomorphic to $\beta \mathbb{R}$.
\end{cor}

\section{Uncountably many compactifications of $\mathbb{R}$}

\noindent Our final task is to show that there are uncountably many different $T_2$ compactifications of $\mathbb{R}$.
\\ For this, we introduce the concept of the \textit{inverse limit} (which really is a limit, in the categorical sense) of a sequence of spaces with maps between them.

\begin{defn}
Suppose that $\langle X_n, d_n \rangle$, for $n \in \mathbb{N}$, is a pair such that $X_n$ is a topological space, and $d_n: X_{n+1} \rightarrow X_n$ is continuous.
\\ The \textbf{inverse limit} $\langle X_\omega, d_{\omega,n} \rangle$ of the sequence $\langle \langle X_n, d_n \rangle: n \in \mathbb{N} \rangle$ is defined as follows. Let $$X_\omega = \{ x \in \prod_{n \in \mathbb{N}} X_n : \forall n, \  x(n) = d_n (x(n+1)) \},$$
and $d_{\omega,n}= \pi_n: X_\omega \rightarrow X_n$ be the restriction of the $n^{th}$ projection to $X_\omega$, so $d_{\omega,n}(x)=\pi_n(x)=x(n)$ for each $x \in X_\omega$.
\end{defn}

\noindent Observe that each $d_{\omega,n}$ is continuous, as the restriction of a continuous function. Observe also that for each $n$, $d_{\omega,n}= d_n \circ d_{\omega,n+1}$, since $$d_n \circ d_{\omega,n+1} (x)= d_n (x(n+1)) = x(n) = d_{\omega,n} (x).$$

\noindent We now give a property that characterises the inverse limit.

\begin{prop} \label{limit}
Suppose $\langle Y, \langle g_n: n\in \mathbb{N} \rangle \rangle$ is any pair such that $Y$ is a topological space, each $g_n: Y \rightarrow X_n$ is continuous, and for all $n$, $g_n = d_n \circ g_{n+1}$. Then there is a continuous function $g: Y \rightarrow X_\omega$ such that for all $n$, $g_n = d_{\omega, n } \circ g$.
\end{prop}

\begin{proof}
Simply define $g: Y \rightarrow X_\omega, \  y \mapsto x$ where $x(n)=g_n(y)$. This is well-defined, because for each $n$ we have $x(n) = d_n (x(n+1))$:
$$x(n) = g_n(y) = d_n \circ g_{n+1} (y)= d_n ( g_{n+1} (y)) = d_n (x(n+1)).$$
We also have $g_n = d_{\omega, n } \circ g$, because
$$d_{\omega, n } \circ g (y)= d_{\omega, n } (x) = x(n) = g_n(y).$$
It only remains to show that $g$ is continuous. We show that the preimage under $g$ of each subbasic open set is open. Let $U= U_j \times \prod_{n \neq j} X_n$, where $U_j$ is open in $X_j$. Then, 
$$g^{-1}(U \cap X_\omega)= \{y \in Y: g(y) \in U\}= \{y \in Y: g_j(y) \in U_j\}=g_j^{-1}(U_j).$$
This is the continuous preimage of an open set, hence it is open.
\end{proof}

\noindent Let us make a few more easy observations.
\\ Firstly, if each $d_n$ is onto, then each $d_{\omega, n}$ is onto. Indeed, given $x_n \in X_n$, we can recursively find $x_i \in X_i$ for each $i>n$ such that $d_i(x_i)=x_{i-1}$, by surjectivity of the $d_i$. We can also define, for $i<n$, $x_i=d_{i+1}(x_{i+1})$. Define $x \in X_\omega$ by $x(n)=x_n$; then $d_{\omega, n}(x)=x(n)=x_n$.
\\

\noindent Also, if all of the spaces $X_n$ are compact Hausdorff, then $X_\omega$ is compact Hausdorff. Indeed, $\prod_{n \in \mathbb{N}} X_n$ is Hausdorff and compact by Tychonoff's theorem, so if we know that $X_\omega$ is a closed subspace, then it is Hausdorff and compact. It remains to see that $X_\omega$ is closed in $\prod_{n \in \mathbb{N}} X_n$. Well, $$X_\omega = \bigcap_{N \in \mathbb{N}} \{ x \in \prod_{n \in \mathbb{N}} X_n : \  x(N) = d_N (x(N+1)) \}= \bigcap_{N \in \mathbb{N}} \psi^{-1}(\Delta_{X_N \times X_N}),$$
where $\psi:\prod_{n \in \mathbb{N}} X_n \rightarrow X_N \times X_N, \ x \mapsto (\pi_N(x), d_N \circ \pi_{N+1}(x))=(x_N,d_N(x(N+1)))$ is continuous, since each component is continuous in $x$. Since $X_N$ is Hausdorff, the diagonal $\Delta_{X_N \times X_N}= \{(x,x): x \in X_N\}$ is closed in $X_N \times X_N$. Therefore each $\psi^{-1}(\Delta_{X_N \times X_N})$ is closed as the continuous preimage of a closed set. Hence $X_\omega$ is closed, as the intersection of closed sets.
\\

\begin{lem} \label{max}
If $\langle \langle g_n, \delta_n \mathbb{R} \rangle: n \in \mathbb{N} \rangle$ is a sequence of metrisasble $T_2$ compactifications of $\mathbb{R}$ such that for all $n$, $\delta_n \mathbb{R} \leq \delta_{n+1} \mathbb{R}$, then there exists a metrisable $T_2$ compactification $\delta_\omega \mathbb{R}$ of $\mathbb{R}$ such that for all $n$, $\delta_n \mathbb{R} \leq \delta_\omega \mathbb{R}$.
\end{lem}

\begin{proof}
By assumption, for each $n$ there is an onto function $d_n: \delta_{n+1} \mathbb{R} \rightarrow \delta_n \mathbb{R}$ such that $d_n \circ g_{n+1} = g_n$. Let us take the inverse limit of the system $\langle \langle \delta_n \mathbb{R}, d_n \rangle: n \in \mathbb{N} \rangle$. Call it $\langle X_\omega, d_{\omega,n} \rangle$. By definition $X_\omega$ is a subspace of a countable product of the spaces $\delta_n \mathbb{R}$, and is therefore metrisable by metrisability of each of the $\delta_n \mathbb{R}$. We have remarked above that $X_\omega$ must be compact Hausdorff, since each individual space $\delta_n \mathbb(R)$ is. Since we have a pair $\langle \mathbb{R}, \langle g_n: n\in \mathbb{N} \rangle \rangle$ such that $\mathbb{R}$ is a topological space, each $g_n: \mathbb{R} \rightarrow \delta_n\mathbb{R}$ is continuous, and for all $n$, $g_n = d_n \circ g_{n+1}$, by Proposition \ref{limit} there is a continuous function $g: \mathbb{R} \rightarrow X_\omega$ such that for all $n$, $g_n = d_{\omega, n } \circ g$. Let $\delta_{\omega} \mathbb{R} = \text{cl}^{X_\omega}(g(\mathbb{R}))$. Then $\langle g, \delta_\omega \mathbb{R} \rangle$ is the desired compactification:

\begin{itemize}
    \item $\delta_\omega \mathbb{R}$ is compact $T_2$ and metrisable, since it is a closed subspace of the compact $T_2$ and metrisable space $X_\omega$;
    \item $g$ is injective since $g_0$ is injective;
    \item $g$ is continuous by assumption;
    \item Suppose $U$ is open in $\mathbb{R}$. We claim $g(U)$ is open in $g(\mathbb{R})$. Well, $d_{\omega, 0} \circ g(U) = g_0 (U)$ is open in $\delta_0 \mathbb{R}$. (It is open in $g_0 (\mathbb{R})$, which is in turn open in $\delta_0 \mathbb{R}$ as $\mathbb{R}$ is locally compact). Then, $d_{\omega, 0}^{-1} (g_0 (U))$ is open in $\delta_\omega \mathbb{R}$, and $g(U)= d_{\omega, 0}^{-1} (g_0 (U))\cap g(\mathbb{R})$ shows that $g(U)$ is open in $g(\mathbb{R})$. Altogether this shows that $g^{-1}$ is continuous;
    \item $\text{cl}^{\delta_{\omega} \mathbb{R}} (g(\mathbb{R})) = \delta_{\omega} \mathbb{R}$ holds.
    \item for all $n$, $\delta_n \mathbb{R} \leq \delta_\omega \mathbb{R}$. This is witnessed by the continuous functions $d_{\omega,n}: \delta_\omega \mathbb{R} \rightarrow \delta_n\mathbb{R}$. We have remarked that they are onto because the $d_n$ are onto; furthermore, for each $n$ we have $g_n = d_{\omega, n } \circ g$.
\end{itemize}
\end{proof}

\noindent We are almost ready to show that $\mathbb{R}$ has uncountably many (non-equivalent) $T_2$ compactifications. For this, let us recall Zorn's Lemma.

\begin{lem}[Zorn's Lemma]
Let $\mathfrak{A} = (A, \leq)$ be a nonempty poset in which every nonempty chain has an upper bound. Then $\mathfrak{A}$ has a maximal element.
\end{lem}

\begin{thm}
$\mathbb{R}$ has uncountably many $T_2$ compactifications.
\end{thm}

\begin{proof}
Suppose $\mathbb{R}$ has only countably many $T_2$ compactifications. In particular $\mathbb{R}$ has only countably many \textit{metrisable} $T_2$ compactifications. We may assume without loss that there are countably infinitely many of these. (If there are only finitely many metrisable $T_2$ compactifications, then certainly one of these is maximal among all the others; this contradicts Lemma \ref{nomax}.)
\\

\noindent Let the set of all metrisable $T_2$ compactifications of $\mathbb{R}$ be $\mathfrak{A}=\{\langle h_n, \delta_n \mathbb{R} \rangle: n \in \mathbb{N} \}$. (We write $\langle h'_n, \delta'_n \mathbb{R} \rangle$ for ease of notation, but we really mean its class $[\langle h'_n, \delta'_n \mathbb{R} \rangle ]$, of course.)
\\ This is a poset. We show that it has a maximal element, by checking that it satisfies the conditions of Zorn's Lemma. $\mathfrak{A}$ is nonempty, since it contains $\langle k,\gamma \mathbb{R} \rangle$. Suppose $\mathfrak{C} \subseteq \mathfrak{A}$ is a nonempty chain. We need to exhibit an upper bound for $\mathfrak{C}$. We split into two cases:

\begin{itemize}
    \item If $\mathfrak{C}$ has only finitely many elements, write these as $\langle h'_1, \delta'_1 \mathbb{R} \rangle  \leq \ldots \leq \langle h'_r, \delta'_r \mathbb{R} \rangle $. Then $ \delta'_r \mathbb{R} $ is a greatest element of the chain, hence certainly an upper bound.
    
    \item If $\mathfrak{C}$ has countably infinitely many elements, write $\mathfrak{C}=\{\langle h'_n, \delta'_n \mathbb{R} \rangle: n \in \mathbb{N} \}$. Let us assume without loss that $\mathfrak{C}$ has no maximal element. (A maximal element in a chain would also be a greatest element and hence an upper bound for the chain, so we would be done.) Note also that each nonempty finite subset $\mathfrak{C'}$ of $\mathfrak{C}$ is still a chain, and by the above case, $\mathfrak{C'}$ has a greatest element max$\{\mathfrak{C}'\}$. We now construct a sequence $\langle \langle h'_{n_i}, \delta'_{n_i} \mathbb{R} \rangle: i \in \mathbb{N} \rangle$ of compactifications in $\mathfrak{C'}$ such that for all $i$, $\delta'_{n_i} \mathbb{R} \leq \delta'_{n_{i+1}} \mathbb{R}$.
    \\
    
    Let $n_0=0$. $\langle h'_{n_0}, \delta'_{n_0} \mathbb{R} \rangle$ is not a maximal element of the chain, so there is $n_1 > n_0$ with $\delta'_{n_0}\mathbb{R}<\delta'_{n_1}\mathbb{R}$.
    \\ For $r >0$, at the $r^{th}$ stage consider the finite subchain $\mathfrak{C}'_r=\{\delta'_i \mathbb{R}: 0 \leq i \leq n_r\}$; max$\{\mathfrak{C}'_r\}$ is not a maximal element of $\mathfrak{C}$, so there is $n_{r+1}>n_r$ with max$\{\mathfrak{C}'_r\} < \delta'_{n_{r+1}}\mathbb{R}$.
    \\
    
    We have inductively defined a sequence $\langle \langle h'_{n_i}, \delta'_{n_i} \mathbb{R} \rangle: i \in \mathbb{N} \rangle$ such that for each $i$, $ \delta'_{n_i} \mathbb{R} \leq  \delta'_{n_{i+1}} \mathbb{R}$. Therefore Lemma \ref{max} applied to this sequence $\langle \langle h'_{n_i}, \delta'_{n_i} \mathbb{R} \rangle: i \in \mathbb{N} \rangle$ (now considered as a sequence of actual compactifications rather than classes of these) tells us that there exists a metrisable $T_2$ compactification $\delta_{\omega} \mathbb{R}$ such that for each $i$, $\delta'_{n_i} \mathbb{R} \leq \delta_{\omega} \mathbb{R}$.
    \\
    
    $\delta_{\omega} \mathbb{R}$ is an element of $\mathfrak{A}$; let us show that it is an upper bound for $\mathfrak{C}$.
    \\ Well, for each $r \in \mathbb{N}$ we have $n_r \geq r$ so $\delta'_r \mathbb{R}$ is among $\delta'_0 \mathbb{R}, \ldots, \delta'_{n_r} \mathbb{R}$. Therefore $$\delta'_r \mathbb{R} \leq \text{max} \{ \mathfrak{C}'_r \}< \delta'_{n_{r+1}} \mathbb{R} \leq \delta_{\omega} \mathbb{R},$$ as required.
\end{itemize}

\noindent We have now shown that $\mathfrak{A}$ satisfies the conditions of Zorn's Lemma, and so has a maximal element $\langle h_{\text{max}},\delta_{\text{max}} \mathbb{R} \rangle$. That is, $\langle h_{\text{max}},\delta_{\text{max}} \mathbb{R} \rangle$ is maximal among all metrisable $T_2$ compactifications of $\mathbb{R}$.
\\

\noindent This contradicts Lemma \ref{nomax}. Therefore $\mathbb{R}$ could not have only countably many $T_2$ compactifications!
\end{proof}

\section{Conclusion}

\noindent In Section 2, for the problem of finding a compactification of $\mathbb{R}$ to which the family $f_n(x)=\cos(x)$ extended continuously, we could have gone a different route by defining $\langle k, \gamma \mathbb{R} \rangle $ as follows. Take $$k: \mathbb{R} \rightarrow [-1,1] \times [-1,1], \  x \mapsto (\tanh(x),\cos(x)),$$ and let $\gamma \mathbb{R}$ be the closure of the image of $k$ in $[-1,1] \times [-1,1]$. Indeed, a bit of thought shows that if we have found a compactification $\langle k, \gamma \mathbb{R} \rangle $ onto which $f_1(x) = \cos(x)$ extends continuously, then for each $n \in \mathbb{Z}$, $f_n(x) = \cos(nx)$ will also extend continuously. 
\\

\noindent This relies on the fact that each $f_n(x)=\cos(nx)$ can be expanded as a polynomial $T_n$ in $\cos(x)$: $$\cos(nx) = T_n (\cos(x)),$$ where, in fact, $T_n$ is the $n^{th}$ Chebyshev polynomial.
\\ Therefore, if we have a compactification $\langle k, \gamma \mathbb{R} \rangle $ and a continuous function $\gamma f_1 : \gamma  \mathbb{R} \rightarrow  \mathbb{R}$ such that $$\gamma_1 f \circ k = f_1,$$ then this would also yield, for each $n \in \mathbb{Z}$, a continuous function $\gamma f_n : \gamma  \mathbb{R} \rightarrow  \mathbb{R}$ such that $$\gamma f_n \circ k = f_n.$$ Simply take $\gamma f_n= T_n \circ \gamma f_1$:
$$\gamma f_n \circ k= T_n \circ \gamma f_1 \circ k = T_n \circ f_1 = f_n.$$

\noindent The advantage of this approach is that we can instantly see this space is metrisable, as a subspace of $[-1,1] \times [-1,1]$. This means we do not need to rely on the result that a countable product of metric spaces is metrisable.
\\

\noindent Notice also that we did not prove that our choice of $\langle k, \gamma \mathbb{R} \rangle $ was smallest among all compactifications to which the family $f_n(x)=\cos(nx)$ extends continuously -- this was not necessary for us to show that $\langle k, \gamma \mathbb{R} \rangle $ is distinct from the one-point, two-point, and Stone-$\check{\text{C}}$ech compactification.

\end{document}